\documentclass[12pt,twoside,reqno]{amsart}
\usepackage{amsxtra,amscd}
\usepackage{graphicx}
\usepackage{amsmath}
\usepackage{amsfonts}
\usepackage{amssymb}

\newtheorem{theorem}{Theorem}[section]
\newtheorem{lemma}[theorem]{Lemma}
\theoremstyle{definition}
\newtheorem{definition}[theorem]{Definition}

\newtheorem{remark}[theorem]{Remark}

\numberwithin{equation}{section}

\begin{document}
\title{on the section conjecture of Grothendieck}
\author{Feng-Wen An}
\address{School of Mathematics and Statistics, Wuhan University, Wuhan,
Hubei 430072, People's Republic of China}
\email{fwan@amss.ac.cn}

\begin{abstract}
For a given arithmetic scheme, in this paper we will introduce and discuss the monodromy action on a universal cover of the \'{e}tale fundamental group and the monodromy action on an \emph{sp}-completion constructed by the graph functor, respectively;  then by these results we will give a proof of the section conjecture of Grothendieck for arithmetic schemes.
\end{abstract}

\maketitle

\begin{center}
{\tiny {Contents} }
\end{center}

{\tiny \qquad {Introduction} }

{\tiny \qquad {1. Statements of the Main Theorems} }

{\tiny \qquad {2. Preliminaries} }

{\tiny \qquad {3. Universal Covers}}

{\tiny \qquad {4. Monodromy Action, I}}

{\tiny \qquad {5. \emph{sp}-Completions}}

{\tiny \qquad {6. Monodromy Action, II}}

{\tiny \qquad {7. \emph{qc} Fundamental Groups}}

{\tiny \qquad {8. Proofs of the Main Theorems}}

{\tiny \qquad {References}}

\section*{Introduction}

The section conjecture in anabelian geometry is originally from \cite{faltings}, the so called \textquotedblleft \emph{Anabelian Letter to Faltings} \textquotedblright, a letter to Faltings written by Grothendieck in 1983.

\textquotedblleft \emph{To Grothendieck's disappointment, Faltings never responded to this letter. However, Faltings' student Shinichi Mochizuki picked up the subject years later and proved Grothendieck's anabelian conjecture for hyperbolic curves} \textquotedblright (see \cite{www}).

See \cite{pop} for a formal introduction to this topic on anabelian geometry.

Many people have proved the section conjecture for the various cases of algebraic curves.

Now let $X$ be an arithmetic scheme surjectively over $Spec(\mathbb{Z})$ of finite type. In this paper we will prove the section conjecture of Grothendieck for arithmetic schemes.

Here are the key points to overcome in the paper:

Naturally and fortunately, we will have the monodromy action of the group $Aut(X_{\Omega_{et}}/X)$  on the universal cover $X_{\Omega_{et}}$ for the \'{e}tale fundamental group $\pi_{1}^{et}(X)$ and the monodromy action of the group $Aut(X_{sp}/X)$  on the \emph{sp}-completion $X_{sp}$ constructed by the graph functor $\Gamma$.

Then by these monodromy actions of groups on integral schemes, we will obtain a bijection between the sets of homomorphisms which are considered.

\textbf{\tiny{Acknowledgment.}} The author would like to express his sincere
gratitude to Professor Li Banghe for his advice and instructions
on algebraic geometry and topology.

\section{Statements of the Main Theorems}

\subsection{Notation}

In this paper, an \textbf{arithmetic variety}  is an integral scheme $X$ satisfying the conditions:
\begin{itemize}
\item $\dim X \geqq 1$.

\item There is a surjective morphism $f:X \to Spec\left( \mathbb{Z}\right) $ of finite type.
\end{itemize}

For a number field $K$ (i.e., a finite extension of $\mathbb{Q}$), let
$\mathcal{O}_{K}$ denote the ring of algebraic integers of $K$.

For an integral scheme $Z$, put
\begin{itemize}
\item $k(Z)\triangleq$ the function field of an integral
scheme $Z;$

\item $\pi _{1}^{et}\left( Z\right) \triangleq$ the \'{e}tale
fundamental group of $Z$ for a geometric point of $Z$ over a
separable closure of the function field $k\left( Z\right).$
\end{itemize}

In particular, for a field $L$, we set $$\pi _{1}^{et}(L)\triangleq\pi _{1}^{et}(Spec(L)).$$

\subsection{Outer homomorphisms}

Let $G,H,\pi_{1},\pi_{2}$ be four groups with homomorphisms $p:G\to \pi_{1}$ and $q:H\to \pi_{2}$, respectively. The \textbf{outer homomorphism set} $Hom_{\pi_{1},\pi_{2}}^{out}(G,H)$ is defined to be the set of the maps $\sigma$ from the quotient $\pi_{1}/p(G)$ into the quotient $\pi_{2}/q(H)$ given by a group homomorphism $f:G\to H$ in such a manner: $$\sigma: x\cdot p(G)\mapsto f(x)\cdot q(H)$$ for any $x\in \pi_{1}$.

In fact, if $G$ and $H$ are normal subgroups of $\pi_{1}$ and $\pi_{2}$, respectively,  $Hom_{\pi_{1},\pi_{2}}^{out}(G,H)$ can be regarded as a subset of $Hom(Out(G),Out(H))$.
Here, $Out(G)\triangleq Aut(G)/Inn(G)$ and $Out(H) \triangleq Aut(H)/Inn(H)$ are the outer automorphism groups.

However, in general, it is not true that $$Hom_{\pi_{1},\pi_{2}}^{out}(G,H)= Hom(Out(G),Out(H))$$ holds.

\subsection{Statements of the main theorems}

For anabelian geometry of arithmetic schemes, we have the following results, which are the main theorems in the present paper.

\begin{theorem}
Let $X$ and $Y$ be two arithmetic varieties such that $k\left( Y\right) $ is contained in $ k\left( X\right)$. Then there is a bijection
\begin{equation*}
Hom\left( X,Y\right) \cong Hom_{\pi _{1}^{et}(k(X)),\pi _{1}^{et}(k(Y))
}^{out}\left( \pi _{1}^{et}\left( X\right) ,\pi _{1}^{et}\left( Y\right)
\right)
\end{equation*}
between sets.
\end{theorem}

\begin{theorem}
Let $X$ be an arithmetic variety and let $K$ be a number field. Suppose that
there is a surjective morphism from $X$ onto $\mathcal{O}_{K}.$ Then there
is a bijection
\begin{equation*}
\Gamma \left( X/Spec\left( \mathcal{O}_{K}\right)\right) \cong Hom_{\pi _{1}^{et}(K),\pi _{1}^{et}(k(X)) }^{out}\left( \pi _{1}^{et}(Spec(\mathcal{O}_{K})) ,\pi _{1}^{et}\left( X\right) \right)
\end{equation*}
between sets.
\end{theorem}

Now fixed a function field $L$ over a number field $K$. Set
\begin{itemize}
\item $G(L)\triangleq$ the absolute Galois group $Gal(L^{al}/L)$;

\item $G(L)^{un}\triangleq$ the Galois group $Gal(L^{un}/L)$ of the maximal unramified extension $L^{un}$ of $L$ (see \emph{Definition 2.8}).
\end{itemize}

Using Galois groups of fields, we  have the following versions of the main theorems above, respectively.

\begin{theorem}
Let $X$ and $Y$ be two arithmetic varieties such that $k\left( Y\right) $ is contained in $ k\left( X\right)$. Then there is a bijection
\begin{equation*}
Hom\left( X,Y\right) \cong Hom_{G(k(X)),G(k(Y))
}^{out}\left( G\left( k(X)\right)^{un} ,G\left( k(Y)\right)^{un}
\right)
\end{equation*}
between sets.
\end{theorem}

\begin{theorem}
Let $X$ be an arithmetic variety and let $K$ be a number field. Suppose that
there is a surjective morphism from $X$ onto $\mathcal{O}_{K}.$ Then there
is a bijection
\begin{equation*}
\Gamma \left( X/Spec\left( \mathcal{O}_{K}\right)\right) \cong Hom_{G(K),G(k(X)) }^{out}\left( G\left( K\right)^{un} ,G\left( k(X)\right)^{un} \right)
\end{equation*}
between sets.
\end{theorem}

We will prove the main theorems above in \emph{\S 8} after preparations are made in \emph{\S\S 2-7}.

\begin{remark}
In a similar manner, we can prove that the two theorems still hold for the case of projective schemes.
\end{remark}

\section{Preliminaries}

\subsection{Convention}

For an integral domain $D$, let $Fr(D)$ denote the field of fractions of $D$. In particular, $Fr(D)$ will be assumed to be contained in $\Omega $ if $D$
is contained in a field $\Omega $.

By an \textbf{integral variety} we will always understand an integral scheme over $Spec(\mathbb{Z})$ by a surjective morphism (not necessarily of finite type).

\subsection{Galois extension}

Let $L$ be an extension of a field $K$. Note that here
 $L$ is not necessarily algebraic over $K$.  Let $Gal(L/K)$ be the Galois group of $L$ over $K$.

Recall that $L$ is said to be \textbf{Galois} over $K$ if
$K$ is the invariant subfield of $Gal(L/K)$, that is, if
$K=\{x\in L :\sigma (x)=x \text{ holds for any }\sigma \in Gal(L/K) \}$.

\subsection{Quasi-galois extension}

Let $L$ be an extension of a field $K$ (not necessarily algebraic).

\begin{definition}
$L$ is said to be \textbf{quasi-galois} over $K$ if each
irreducible polynomial $f(X)\in F[X]$ that has a root in $L$ factors
completely in $L\left[ X\right] $ into linear factors for any subfield $F$ with $K\subseteq F\subseteq L$.
\end{definition}

Let $D\subseteq D_{1}\cap D_{2}$ be three integral domains. Then $D_{1}$
is said to be \textbf{quasi-galois} over $D$ if $Fr\left(
D_{1}\right) $ is quasi-galois over $Fr\left( D\right) $.

\begin{definition}
$D_{1}$ is said to be a \textbf{conjugation} of $D_{2}$ over $D$ if
there is an $F-$isomorphism $\tau:Fr(D_{1})\rightarrow
Fr(D_{2})$ such that
$
\tau(D_{1})=D_{2}$,
where $F\triangleq k(\Delta)$, $k\triangleq Fr(D)$, $\Delta$ is a transcendental basis of the field $Fr(D_{1})$ over $k$, and $F$ is contained in $Fr(D_{1})\cap Fr(D_{2})$.
\end{definition}

\subsection{Affine covering with values}

Let $X$ be a scheme. An \textbf{affine covering} of $X$ is a family $
\mathcal{C}_{X}=\{(U_{\alpha },\phi _{\alpha };A_{\alpha })\}_{\alpha \in
\Delta }$, where for each $\alpha \in \Delta $, $\phi _{\alpha }$ is an
isomorphism from an open set $U_{\alpha }$ of $X$ onto the spectrum $Spec{
A_{\alpha }}$ of a commutative ring $A_{\alpha }$.

Each $(U_{\alpha },\phi _{\alpha };A_{\alpha })\in \mathcal{C}_{X}$ is
called a \textbf{local chart}. For the sake of brevity, a local chart $
(U_{\alpha},\phi_{\alpha};A_{\alpha })$ will be denoted by $U_{\alpha}$ or $
(U_{\alpha},\phi_{\alpha})$.

An affine covering $\mathcal{C}_{X}$ of $(X, \mathcal{O}_{X})$ is said to be
\textbf{reduced} if $U_{\alpha}\neq U_{\beta} $ holds for any $\alpha\neq
\beta$ in $\Delta$.

Let $\mathfrak{Comm}$ be the category of commutative rings with identity.
For a given field $\Omega$, let $\mathfrak{Comm}(\Omega)$ be the category
consisting of the subrings of $\Omega$ and their isomorphisms.

\begin{definition}
Let $\mathfrak{Comm}_{0}$ be a subcategory of $\mathfrak{Comm}$. An affine
covering $\{(U_{\alpha},\phi_{\alpha};A_{\alpha })\}_{\alpha \in \Delta}$ of
$X$ is said to be \textbf{with values} in $\mathfrak{Comm}_{0}$ if for each $
\alpha \in \Delta$ there are $\mathcal{O}_{X}(U_{\alpha})=A_{\alpha}$ and $U_{\alpha}=Spec(A_{\alpha})$, where
 $A_{\alpha }$ is a ring contained in $\mathfrak{Comm}_{0}$.

In particular, an affine covering $\mathcal{C}_{X}$
of $X$ with values in $\mathfrak{Comm}(\Omega)$ is said to be \textbf{with
values in the field $\Omega$}.
\end{definition}

By an affine covering with values in a field, it is seen that an affine open set of a scheme is measurable, at the same time, the non-affine open sets are unmeasurable.

If we ignore the non-affine open sets, almost all properties of the scheme will be still preserved. Hence, we have the following notions.

Let $\mathcal{O}_{X}$ and $\mathcal{O}^{\prime}_{X}$ be two structure sheaves on the underlying space of an integral scheme $X$. The two integral schemes $(X,\mathcal{O}_{X})$ and $(X, \mathcal{O}^{\prime}_{X})$ are said to be \textbf{essentially equal} provided that for any open set $U$ in $X$, we have
 $$U \text{ is affine open in }(X,\mathcal{O}_{X}) \Longleftrightarrow \text{ so is }U \text{ in }(X,\mathcal{O}^{\prime}_{X})$$ and in such a case,  $D_{1}=D_{2}$ holds or  there is $Fr(D_{1})=Fr(D_{2})$ such that for any nonzero $x\in Fr(D_{1})$, either $$x\in D_{1}\bigcap D_{2}$$ or $$x\in D_{1}\setminus D_{2} \Longleftrightarrow x^{-1}\in D_{2}\setminus D_{1}$$ holds, where $D_{1}=\mathcal{O}_{X} (U)$ and $D_{2}=\mathcal{O}^{\prime}_{X} (U)$.

 Two schemes $(X,\mathcal{O}_{X})$ and $(Z,\mathcal{O}_{Z})$ are said to be \textbf{essentially equal} if the underlying spaces of $X$ and $Z$ are equal and the schemes $(X,\mathcal{O}_{X})$ and $(X,\mathcal{O}_{Z})$ are essentially equal.

\begin{definition}
An affine covering $\{(U_{\alpha },\phi _{\alpha };A_{\alpha })\}_{\alpha
\in \Delta }$ of $X$ is said to be an \textbf{affine
patching} of $X$ if the map $\phi _{\alpha }$ is the
identity map on $U_{\alpha }=SpecA_{\alpha }$ for each $\alpha \in \Delta .$
\end{definition}

Evidently, an affine patching is reduced.

\subsection{Quasi-galois closed affine covering}

Let $f:X\rightarrow Y$ be a
surjective morphism of integral schemes. Fixed an algebraic closure $\Omega$ of the function field $k(X)$.

\begin{definition}
A reduced affine covering $\mathcal{C}_{X}$ of $X$ with values in $\Omega $
is said to be \textbf{quasi-galois closed} over $Y$ by $f$ if there exists a local chart $(U_{\alpha }^{\prime },\phi _{\alpha }^{\prime
};A_{\alpha }^{\prime })\in \mathcal{C}_{X}$ such that $U_{\alpha }^{\prime
}\subseteq \varphi^{-1}(V_{\alpha})$ holds
\begin{itemize}
\item for any $(U_{\alpha },\phi _{\alpha
};A_{\alpha })\in \mathcal{C}_{X}$;

\item for any affine open set $V_{\alpha}$ in $
Y$ with $U_{\alpha }\subseteq f^{-1}(V_{\alpha})$;

\item for any
conjugate $A_{\alpha }^{\prime }$ of $A_{\alpha }$ over $B_{\alpha}$,
\end{itemize}
 where $
B_{\alpha}$ is the canonical image of $\mathcal{O}_{ Y}(V_{\alpha})$ in  $k(X)$ via $f$.
\end{definition}

\subsection{Quasi-galois closed scheme}

Let  $f:X\rightarrow Y$ be a
surjective morphism between integral schemes. Let $Aut\left( X/Y\right) $ denote the group of
automorphisms of $X$ over $Y$.

A integral scheme $Z$ is said to be a \textbf{conjugate} of $X$ over $Y$ if
there is an isomorphism $\sigma :X\rightarrow Z$ over $Y$.

\begin{definition}
$X$ is said to be \textbf{quasi-galois closed} over $Y$ by $f$ if there is an algebraically closed field $\Omega$
and a reduced affine covering $\mathcal{C}_{X}$ of $X$ with values in $
\Omega $ such that for any conjugate $Z$ of
$X$ over $Y$ the two conditions are satisfied:
\begin{itemize}
\item $(X,\mathcal{O}_{X})$ and $(Z,\mathcal{O}_{Z})$ are essentially equal if $Z$ has a reduced
affine covering with values in $\Omega$.

\item $\mathcal{C}_{Z}\subseteq \mathcal{C}_{X}$ holds if $\mathcal{C}_{Z}$
is a reduced affine covering of $Z$ with values in $\Omega $.
\end{itemize}
\end{definition}

\begin{remark}
In the above definition, $\Omega$ is in deed an algebraic closure of the function field $k(X)$; $\mathcal{C}_{X}$ is the unique maximal affine covering of $X$ with values in $\Omega$ (see \cite{An7}).
\end{remark}

\subsection{Unramified extension}

Let us recall the definition for unramified extensions of function fields over a number field in several variables.

\begin{definition}
Let $L_{1}$ and $L_{2}$ be two extensions over a number field $K$
such that $L_{1}\subseteq L_{2}$.

$\left( i\right) $ $L_{2}$ is said to be a \textbf{finite unramified Galois}
extension of $L_{1}$ if there are two arithmetic varieties $X_{1}$ and $
X_{2} $ and a surjective morphism $f:X_{2}\rightarrow X_{1}$ such that

\begin{itemize}
\item $k\left( X_{1}\right) =L_{1},k\left( X_{2}\right) =L_{2}$;

\item $X_{2}$ is a finite \'{e}tale Galois cover of $X_{1}$ by $f$.
\end{itemize}

$\left( ii\right) $ $L_{2}$ is said to be a \textbf{finite unramified}
extension of $L_{1}$ if there is a field $L_{3}$ over $K$ such that $L_{2}$
is contained in $L_{3}$ and $L_{3}$ is a finite unramified Galois extension
of $L_{1}.$

$\left( iii\right) $ $L_{2}$ is said to be an \textbf{unramified} extension
of $L_{1}$ if the field $L_{1}(\omega)$ is a finite unramified extension of $
L_{1}$ for each element $\omega \in L_{2}$. In such a case, the element $
\omega$ is said to be \textbf{unramified} over $L_{1}$.
\end{definition}

\begin{remark}
It is seen that there exists the geometric model $X_{2}/X_{1}$ for the extension $L_{2}/L_{1}$ (see \cite{An3}). In deed, we can take the valuation rings $A_{1}\subseteq A_{2}$ of $L_{1}\subseteq L_{2}$, respectively; then put $X_{1}=Spec(A_{1})$ and $X_{2}=Spec(A_{2})$.
\end{remark}

\begin{remark} Let $L_{1}\subseteq L_{2} \subseteq L_{3}$ be function field over a number field $K$. Suppose that $L_{2}/L_{1}$ and $L_{3}/L_{2}$ are unramified extensions. Then $L_{3}$ is unramified over $L_{1}$.
\end{remark}

\begin{remark}
It is seen that for the case of an algebraic extension, the unramified extension defined in \emph{Definition 2.8} coincides exactly  with that in algebraic number theory.
\end{remark}

\begin{remark}
Note that we have defined another unramified extension in \cite{An6} for the case of algebraic schemes, which is only a formally abstract definition and is different from the above one in \emph{Definition 2.8}.
\end{remark}

Let $L$ be an arbitrary extension over a number field $K$. Set

\begin{itemize}
\item $L^{al}\triangleq $ an algebraical closure of $L$;

\item $L^{un}\triangleq $ the union of all the finite unramified
subextensions over $L$ contained in $L^{al}$.
\end{itemize}

\section{Universal Covers}

\subsection{Facts on quasi-galois closed schemes.}

 Here there are several known results on quasi-galois schemes which will be used in the remainder of the paper  (see \cite{An2}-\cite{An7}).

\begin{lemma}\emph{(\textbf{Tuning scheme} \cite{An4*})}
For any integral variety $X$, there is an integral variety $Z$
satisfying the conditions:

\begin{itemize}
\item $k\left( X\right) =k\left( Z\right);$

\item $X\cong Z$ are isomorphic;

\item $Z$ has a reduced affine covering with values in $k(X)^{al}$.
\end{itemize}
\end{lemma}

\begin{lemma}
\emph{(\textbf{Geometric model} \cite{An5})}
Let $f:X\rightarrow Y$ be a surjective morphism of integral varieties.
Suppose that $X$ is quasi-galois closed over $Y$ by $f$ and that $k\left( X\right) $
is canonically Galois over $k\left( Y\right) .$ Then $f$ is
affine and there is a group isomorphism
\begin{equation*}
Aut\left( X/Y\right) \cong Gal\left( k\left( X\right) /k\left( Y\right)
\right) .
\end{equation*}
\end{lemma}

\begin{lemma}
\emph{(\textbf{Quotient} \cite{An5})}
Let $X$ and $Y$ be integral varieties such that $X$ is quasi-galois closed
over $Y$ by a surjective morphism $\phi $. Then there is a natural
isomorphism
\begin{equation*}
\mathcal{O}_{Y}\cong \phi _{\ast }(\mathcal{O}_{X})^{{Aut}\left( X/Y\right) }.
\end{equation*}

Here $(\mathcal{O}_{X})^{{Aut}\left( X/Y\right) }(U)$ denotes the invariant
subring of $\mathcal{O}_{X}(U)$ under the natural action of ${Aut}\left(
X/Y\right) $ for any open subset $U$ of $X$.
\end{lemma}

\begin{lemma}
\emph{(\textbf{Geometric model} \cite{An2,An2*})}
Let $X$ and $Y$ be arithmetic
varieties such that $X$ is quasi-galois closed over $Y$ by a
surjective morphism $f$ of finite type. Then
\begin{itemize}
\item $f$ is affine;

\item $k\left( X\right) $ is canonically Galois
over $k(Y)$;

\item there is a group isomorphism
$$
{Aut}\left( X/Y\right) \cong Gal(k\left( X\right) /k(Y)).
$$
\end{itemize}

In particular, let $\dim X=\dim Y$. Then
$X$ is a pseudo-galois cover of \, $Y$ in the sense of Suslin-Voevodsky \emph{(see \cite{sv1,sv2} for definition)}.
\end{lemma}

\begin{lemma}
\emph{(\textbf{Criterion})}
Let $X,Y$ be integral schemes and let $f :X\rightarrow Y$ be a surjective morphism.
Suppose that the function field $k(Y)$ is contained in $\Omega$. The following statements are equivalent:
\begin{itemize}
\item The scheme $X$ is
quasi-galois closed over $Y$ by $f$.

\item There is a unique maximal affine
patching $\mathcal{C}_{X}$ of $X$ with values in $\Omega $ such that $\mathcal{C}_{X}$ is quasi-galois closed over $Y$ by $f$.
\end{itemize}
\end{lemma}

\begin{proof}
It is easily proved in a manner similar to \cite{An7}.
\end{proof}

\subsection{A universal cover for the \'{E}tale fundamental group}

For convenience, let's recall the universal cover for an \'{e}tale fundamental group of an arithmetic variety.

Fixed an arithmetic variety $X$. Let $\Omega$ be an algebraic closure of the function field $k(X)$.
In the following we will construct an integral variety $X_{\Omega _{et}}$ and a morphism $p_{X}:X_{\Omega _{et}} \to X$ such that $X_{\Omega _{et}}$ is quasi-galois closed over $X$.

For brevity, put $K=k\left( X\right) $ and $L=K^{un}\subseteq \Omega .$
By \emph{Lemma 3.1}, without loss of generality, assume that $X$ has a reduced affine covering $\mathcal{C}_{X}$ with values in $\Omega .$
We choose $\mathcal{C}_{X}$ to be maximal (in the sense of set inclusion).

We will proceed in several steps:
\begin{itemize}
\item Fixed a set $\Delta $ of generators of the field $L$
over $K$. Put $G=Gal\left( L/K\right) .$

\item For any local chart $\left( V,\psi
_{V},B_{V}\right) \in \mathcal{C}_{X}$, define $A_{V}=B_{V}\left[ \Delta _{V}\right] $, i.e., the subring of $L$
generated over $B_{V}$ by the set
$
\Delta _{V}=\{\sigma \left( x\right) \in L:\sigma \in G,x\in \Delta \}.
$
Set $i_{V}:B_{V}\rightarrow A_{V}$ to be the inclusion.

\item Define
$$
\Sigma =\coprod\limits_{\left( V,\psi _{V},B_{V}\right) \in \mathcal{C}
_{X}}Spec\left( A_{V}\right)
$$
to be the disjoint union. Let $\pi _{X}:\Sigma \rightarrow X$ be the
projection induced by the inclusions $i_{V}$.

\item Define an equivalence relation $R_{\Sigma }$ in $
\Sigma $ in such a manner:

For any $x_{1},x_{2}\in \Sigma $, we say $x_{1}\sim x_{2}$ if and only if $
j_{x_{1}}=j_{x_{2}}$ holds in $L$.

Here $j_{x}$ denotes the corresponding
prime ideal of $A_{V}$ to a point $x\in Spec\left( A_{V}\right) $.

Let $ X_{\Omega _{et}}$ be the quotient space $\Sigma /\sim $ and let $\pi_{\Omega_{et}}:\Sigma \rightarrow  X_{\Omega _{et}}$ be the
projection of spaces.

\item Define a map $p_{X}:X_{\Omega_{et}}\rightarrow X$ of spaces  by
$
\pi_{\Omega_{et}}\left( z\right) \longmapsto \pi _{X}\left( z\right)
$
for each $z\in \Sigma $.

\item Define
\begin{equation*}
\mathcal{C}_{X_{\Omega_{et}}}=\{\left( U_{V},\varphi _{V},A_{V}\right) \}_{\left( V,\psi
_{V},B_{V}\right) \in \mathcal{C}_{X}}
\end{equation*}%
where $U_{V}=\pi _{X}^{-1}\left( V\right) $ and $\varphi
_{V}:U_{V}\rightarrow Spec(A_{V})$ is the identity map for each $\left(
V,\psi _{V},B_{V}\right) \in \mathcal{C}_{X}$.

Hence, there is a scheme, namely $X_{\Omega_{et}}$,
obtained by gluing the affine
schemes $Spec\left( A_{V}\right) $ for all $\left( U_{V},\varphi
_{V},A_{V}\right) \in \mathcal{C}_{X}$ with respect to the equivalence
relation $R_{\Sigma }$ (see \cite{EGA,Hrtsh}). Naturally, $p_{X}$ becomes a morphism of schemes.
\end{itemize}
This completes the construction.

It follows that we have the following lemma.

\begin{lemma}
\emph{\textbf{(Universal cover)}}
For an arithmetic variety $X$, there is an integral variety $X_{\Omega
_{et}}$ and a surjective morphism $p_{X}:X_{\Omega _{et}}\rightarrow X$
satisfying the conditions:

\begin{itemize}
\item $k\left( X_{\Omega _{et}}\right) ={k(X)}^{un}$;

\item $p_{X}$ is affine;

\item $k\left( X_{\Omega _{et}}\right) $ is Galois over $k\left( X\right) ;$

\item $X_{\Omega _{et}}$ is quasi-galois closed over $X$ by $p_{X}$.
\end{itemize}
\end{lemma}

Such an integral variety $X_{\Omega _{et}}$ is called a \textbf{universal
cover} over $X$ for the \'{e}tale fundamental group $\pi _{1}^{et}\left( X\right) $, denoted by $
\left( X_{\Omega _{et}},p_{X}\right) .$

\begin{proof}
Let $K$ be the function field $k(X)$. Take any $\omega$ in the field $K^{un}$. By \emph{Definition 2.4} and \emph{Lemma 3.4} it is easily seen that every conjugate of $\omega$ over $K$ is also contained in $K^{un}$. This proves that $K^{un}$ is a Galois extension of $K$.

It is seen that $\mathcal{C}_{X_{\Omega_{et}}}$ is the unique maximal affine patching of the scheme $X_{\Omega _{et}}$. From \emph{Lemma 3.5} it is seen that $X_{\Omega _{et}}$ is quasi-galois closed over $X$.
Then it is immediate from  \emph{Lemma 3.2}.
\end{proof}

\section{Monodromy Action, I}

We have the following computations of the \'{e}tale fundamental group of an arithmetic variety.

\begin{lemma}
\emph{(\cite{An4*})}
Fixed any arithmetic variety $X$. Then there exists an isomorphism
$$
\pi _{1}^{et}\left( X\right) \cong Gal\left( {k(X)}^{un}/k\left( X\right)
\right).
$$
\end{lemma}

\begin{lemma}
For any arithmetic variety $X$, there is an isomorphism
$$Aut\left( X_{\Omega _{et}}/X\right)
 \cong \pi _{1}^{et}\left( X\right)$$
 where $\left( X_{\Omega _{et}},p_{X}\right)$ is a universal cover for the group $\pi_{1}^{et}(X)$.
\end{lemma}

\begin{proof}
As $ k\left( X_{\Omega _{et}}\right)=k(X)^{un}$ is a Galois extension over $k(X)$, by \emph{Lemma 3.2} we have $Aut\left( X_{\Omega _{et}}/X\right) \cong Gal\left( k\left( X_{\Omega
_{et}}\right) /k\left( X\right) \right)$. Then it is immediate from \emph{Lemma 4.1} above.
\end{proof}

Now let $X$ and $Y$ be two arithmetic varieties such that $k\left( Y\right) $ is
contained in an algebraic closure $\Omega $ of the function field $k\left( X\right) .$

Fixed a group homomorphism $$\sigma :\pi _{1}^{et}\left( X\right) \rightarrow
\pi _{1}^{et}\left( Y\right) .$$ By \emph{Lemma 4.2} we have a group homomorphism, namely $$
\sigma :Aut\left( X_{\Omega _{et}}/X\right) \rightarrow Aut\left( Y_{\Omega
_{et}}/Y\right) .$$

\begin{lemma}
\emph{\textbf{(Monodromy action)}}
Assume that $\left( X_{\Omega _{et}},p_{X}\right)$ and $\left( Y_{\Omega _{et}},p_{Y}\right)$ are the universal covers for the groups $\pi_{1}^{et}(X)$ and $\pi_{1}^{et}(Y)$, respectively.
Fixed a group homomorphism $$\sigma :Aut\left( X_{\Omega
_{et}}/X\right) \rightarrow Aut\left( Y_{\Omega _{et}}/Y\right) .$$ Then
there is a bijection $$\tau :Hom\left( X,Y\right) \to Hom\left( X_{\Omega _{et}},Y_{\Omega _{et}}\right) , f\mapsto f_{et}$$ between sets given in a canonical manner:
\begin{itemize}
\item Let $f\in Hom\left( X,Y\right)$. Then the map $$g\left( x_{0}\right)
\longmapsto \sigma \left( g\right) \left( h\left( x_{0}\right) \right) $$
defines a morphism $$f_{et}:X_{\Omega _{et}}\rightarrow Y_{\Omega _{et}}$$
for any $x_{0}\in X$ and any $g\in Aut\left( X_{\Omega _{et}}/X\right) .$

\item Let $f_{et}\in Hom\left( X_{\Omega _{et}},Y_{\Omega _{et}}\right)$. Then the map $$
p_{X}\left( x\right) \longmapsto p_{Y}\left( f_{et}\left( x\right) \right)
$$ defines a morphism $$f:X\rightarrow Y$$ for any $x\in X_{\Omega _{et}}.$
\end{itemize}
In particular, we have $$f\circ p_{X}=p_{Y}\circ f_{et}.$$
\end{lemma}

\begin{proof}
It is immediate from \emph{Lemmas 3.2-3,4.2}.
\end{proof}

\section{\emph{sp}-Completion}

In this section we will use Weil's theory of specializations (see \cite{An1} for detail) to give the completion of rational maps between schemes.

\subsection{Definition for specializations}

Let $E$ be a topological
space $E$ and $x,y\in E$. If $y$ is in the closure $\overline{\{x\}}$, $y$ is said to be a \textbf{
specialization} of $x$ (or, $x$ is said to be a \textbf{generalization} of $
y $) in $E$, denoted by $x\rightarrow y$. Put $Sp\left( x\right)
=\{y\in E\mid x\rightarrow y\}$.
 It is evident that $Sp\left( x\right) =\overline{\{x\}}$ is an irreducible closed
 subset in $E$.

 If $x\rightarrow
y$ and $y\rightarrow x$ both hold in $E$, $y$ is said to be a \textbf{generic
specialization} of $x$ in $E$, denoted by $x\leftrightarrow y$. The point $x$
is said to be \textbf{generic} (or \textbf{initial}) in $E$ if we have $
x\leftrightarrow z$ for any $z\in E$ such that $z\rightarrow x$. And $x$ is said to be
\textbf{closed} (or \textbf{final}) if we have $x\leftrightarrow z$ for any $
z\in E$ such that $x\rightarrow z.$ We say that $y$ is a
\textbf{closest specialization}
of $x$ in $X$ if either $z=x$ or $z=y$ holds for any $z\in X$ such that $
x\rightarrow z$ and $z\rightarrow y$.

\subsection{Any specialization is contained in an affine open set}

Let $E=Spec\left( A\right) $ be an affine scheme. For any point
$z\in Spec\left( A\right)$, denote by $j_{z}$ the corresponding
prime ideal in $A$. Then we have a specialization
$x\rightarrow y$ in $Spec\left( A\right) $ if and only if
$j_{x}\subseteq j_{y}$ holds in $A$. Hence, there is a generic
specialization $x\leftrightarrow y$ in $Spec\left( A\right) $ if and
only if $x=y$ holds.

Now consider a scheme $X$.

\begin{lemma}
\emph{(\cite{An1})}
For any points $x,y \in X$, we have $x\leftrightarrow y$ in $X$ if and
only if $x=y$.
\end{lemma}

\begin{proof}
$\Leftarrow$. Trivial. Prove $\Rightarrow$. Assume $x\leftrightarrow
y$ in $X$. Let $U$ be an affine open set of $X$ containing $x$.
From $x\leftrightarrow y$ in $X$, we have ${Sp}(x)={Sp}(y)$; then $x\in{Sp}
(x)\bigcap U={Sp}(y)\bigcap U \ni y$; hence, $x\leftrightarrow y $
in $U$. It follows that $x=y$ holds in $U$ (and of course in $
X$).
\end{proof}

\begin{lemma}
\emph{(\cite{An1})}
Fixed any specialization $x\rightarrow y$ in $X$. Then there is an
affine open subset $U$ of $X$ such that the two points $x$ and $y$ are both
contained in $U$.

In particular, any affine open set in $X$
containing the specialization $y$ must contain the generalization $x$.
\end{lemma}

\begin{proof}
Assume $x \not= y$. Then
$y$ is a limit point of the one-point set $\{x\}$ since $y$ is
contained in the topological closure $Sp(x)$ of $\{x\}$. Let
$U\subseteq X$ be an open set containing $y$. We have $U \bigcap
(\{x\}\setminus \{y\})\neq \emptyset$ by the definition for a limit
point of a set (see any standard textbook for general topology). We
choose $U$ to be an affine open set of $X$.
\end{proof}

\subsection{Any morphism preserves specializations}

Let $f:E\rightarrow F$ be a map of spaces. The map $f$ is said to be \textbf{
specialization-preserving} if there is a specialization $f\left(
x\right) \rightarrow f\left( y\right) $ in $F$ for any
specialization $x\rightarrow y$ in $E$.

\begin{lemma}
\emph{(\cite{An1})}
Any morphism between schemes is specialization-preserving.
\end{lemma}

\begin{proof}
It is immediate from  \emph{Lemma 5.2}.
\end{proof}

\subsection{The graph functor $\Gamma$ from schemes to graphs}

We have such a covariant functor from the category of schemes to
the category of (combinatorial) graphs. See \cite{tut} for preliminaries on graph theory.

\begin{lemma}
\emph{(\textbf{Graph functor} \cite{An1})}
There exists a covariant functor $\Gamma$ from the category $Sch$ of
schemes to the category $Grph$ of graphs given in such a natural
manner.
\begin{itemize}
\item To any scheme $X$, assign the graph $\Gamma(X)$ in which the
vertex set is the set of points in the underlying space $X$ and the
edge set is the set of specializations in $X$.

Here, for any points $x,y \in X$, we say that there is an edge from
$x$ to $y$ if and only if there is a specialization $x\rightarrow y$
in $X$.

\item To any scheme morphism $f:X\rightarrow Y$, assign the graph
homomorphism $\Gamma(f): \Gamma(X)\rightarrow \Gamma(Y)$.

Here, any specialization $x \rightarrow y$ in the scheme $X$ as
an edge in $\Gamma(X)$, is mapped by $\Gamma(f)$ into the
specialization $f(x)\rightarrow f(y)$ as an edge in $\Gamma(Y)$.
\end{itemize}
\end{lemma}

\begin{proof}
It is immediate from \emph{Lemmas 5.2-3}.
\end{proof}

The above functor $\Gamma: Sch\to Grph$ is said to be the \textbf{graph functor}.

\begin{remark}
There are many beautiful  graphs $\Gamma(X)$ associated with schemes $X$. For example, it is easily seen that
\begin{itemize}
\item $\Gamma(Spec(\mathbb{Z}))$  is a star-shaped graph;

\item $\Gamma(Spec(\mathbb{Z}[t]))$ is a graph of
infinitely many loops.
\end{itemize}
\end{remark}

\begin{remark}
By the graph functor $\Gamma$, many invariants that are defined on graphs can be introduced into schemes in a natural manner, for example,
the discrete Morse theory, the Kontsevich's graph
homology theory, etc.
\end{remark}

\subsection{\emph{sp}-completion}

In virtue of the graph functor $\Gamma$, we can give the completion
of a rational maps between schemes, which will be applied to the proofs of the main theorems of the paper.

Let's recall basic definitions for graphs (see \cite{tut}).
Fixed a graph $X$. Let $V(X)$ be the \textbf{set of vertices} in $X$ and  $E(X)$ the \textbf{set of edges} in $X$.

Let $Y$ be a graph. Then $Y$ is said to be a \textbf{subgraph} of $X$ if the following conditions are satisfied:
\begin{itemize}
\item $V(X)\supseteq V(Y)$;

\item $E(X)\supseteq E(Y)$;

\item Every $L\in E(Y)$ has the same ends in $Y$ as in $X$.
\end{itemize}

Recall that an \textbf{isomorphism} $t$ from $X$ onto $Y$ is a ordered pair $(t_{V},t_{E})$ satisfying the conditions:
\begin{itemize}
\item $t_{V}$ is a bijection from $V(X)$ onto $V(Y)$;

\item $t_{E}$ is a bijection from $V(E)$ onto $V(E)$;

\item Let $x \in V(X)$ and $L\in E(X)$. Then $x $ is incident with $L$ if and only if $t_{V}(x) \in V(Y)$ is incident with $t_{E}(L)\in E(Y)$.
\end{itemize}

Now we consider the graphs of integral schemes.

\begin{definition}
An integral scheme $X$ is said to be \textbf{$sp-$complete} if $X$ and $Y$ must be essentially equal for any integral scheme $Y$ such that
\begin{itemize}
\item $\Gamma (X)$ is isomorphic to a subgraph of $\Gamma (Y)$;

\item $k(Y)$ is contained in an algebraic closure of $k(X)$.
\end{itemize}
\end{definition}

\begin{remark}
Let $X$ be an $sp-$complete integral variety. It is easily seen that the function field $k(X)$ must be algebraically closed. In such a case, the graph $\Gamma(X)$ is maximal (by set-inclusion).

For example, let $\mathcal{O}$ be the set of all algebraic numbers over $\mathbb{Q}$. Then $Spec(\mathbb{Z}[\mathcal{O}])$ is $sp-$complete.
\end{remark}

For the $sp-$complete, we have the following theorem.

\begin{theorem}
\emph{\textbf{($sp-$completion)}}
For any integral variety $X$, there exists an integral variety $X_{sp}$ and a surjective morphism $\lambda_{X}:X_{sp}\to X$ such that
\begin{itemize}
\item $\lambda_{X}$ is affine;

\item $X_{sp}$ is $sp-$complete;

\item $k(X_{sp})$ is an algebraic closure of $k(X)$;

\item $X_{sp}$ is quasi-galois closed over $X$ by $\lambda_{X}$.
\end{itemize}
\end{theorem}

Such an integral scheme $X_{sp}$, is said to be an \textbf{$sp-$completion} of $X$. We will denote this by $(X_{sp},\lambda_{X})$.

\begin{proof}
Let $K=k\left( X\right) $ and $L=K^{al}$. Fixed a set $\Delta $ of generators of the field $L$
over $K$. Put $G=Gal\left( L/K\right) .$
By \emph{Lemma 3.1}, without loss of generality, assume that $X$ has a reduced affine covering $\mathcal{C}_{X}$ with values in $\Omega .$
We choose $\mathcal{C}_{X}$ to be maximal (in the sense of set inclusion).

We will proceed in several steps to give the construction:
\begin{itemize}
\item For any local chart $\left( V,\psi
_{V},B_{V}\right) \in \mathcal{C}_{X}$, define $A_{V}=B_{V}\left[ \Delta _{V}\right] $, where
$
\Delta _{V}=\{\sigma \left( x\right) \in L:\sigma \in G,x\in \Delta \}.
$
Set $i_{V}:B_{V}\rightarrow A_{V}$ to be the inclusion.

\item Define
$$
\Sigma =\coprod\limits_{\left( V,\psi _{V},B_{V}\right) \in \mathcal{C}
_{X}}Spec\left( A_{V}\right)
$$
to be the disjoint union. Let $\pi _{X}:\Sigma \rightarrow X$ be the
projection induced by the inclusions $i_{V}$.

\item Define an equivalence relation $R_{\Sigma }$ in $
\Sigma $ in such a manner:

For any $x_{1},x_{2}\in \Sigma $, we say $x_{1}\sim x_{2}$ if and only if $
j_{x_{1}}=j_{x_{2}}$ holds in $L$.

Here $j_{x}$ denotes the corresponding
prime ideal of $A_{V}$ to a point $x\in Spec\left( A_{V}\right) $.

Let $ X_{sp}$ be the quotient space $\Sigma /\sim $ and let $\pi_{sp}:\Sigma \rightarrow  X_{sp}$ be the
projection of spaces.

\item Define a map $\lambda_{X}:X_{sp}\rightarrow X$ of spaces  by
$
\pi_{sp}\left( z\right) \longmapsto \pi _{X}\left( z\right)
$
for each $z\in \Sigma $.

\item Define
\begin{equation*}
\mathcal{C}_{X_{sp}}=\{\left( U_{V},\varphi _{V},A_{V}\right) \}_{\left( V,\psi
_{V},B_{V}\right) \in \mathcal{C}_{X}}
\end{equation*}%
where $U_{V}=\pi _{X}^{-1}\left( V\right) $ and $\varphi
_{V}:U_{V}\rightarrow Spec(A_{V})$ is the identity map for each $\left(
V,\psi _{V},B_{V}\right) \in \mathcal{C}_{X}$.

Hence, we obtain a scheme, namely $X_{sp}$,
by gluing the affine
schemes $Spec\left( A_{V}\right) $ for all $\left( U_{V},\varphi
_{V},A_{V}\right) \in \mathcal{C}_{X}$ with respect to the equivalence
relation $R_{\Sigma }$. Naturally, $\lambda_{X}$ becomes a morphism of schemes.
\end{itemize}
Evidently, it needs only to verify that $X_{sp}$ is $sp-$complete.

In deed, take any integral scheme $Y$ such that
\begin{itemize}
\item $\Gamma (X)$ is isomorphic to a subgraph of $\Gamma (Y)$;

\item $k(Y)$ is contained in an algebraic closure of $k(X)$.
\end{itemize}

 Hypothesize that there is some $x_{0}\in Y \setminus X_{sp}$. Let $z_{0}\in Y$ be final. By \emph{Lemma 5.2} it is seen that $z_{0}$ is not contained in $X_{sp}$. Evidently, there is an affine open set $W=Spec(A_{0})$ in $Y$ such that $z_{0}\in W$.

 On the other hand, there must be an affine open set $W_{sp}=Spec(B_{0})$ in $X_{sp}$ such that the ring $A_{0}$ is contained in $B_{0}$ from the construction above. It follows that $z_{0}$ must be contained in $W_{sp}$ and hence in $X_{sp}$, which will be in contradiction.
\end{proof}

The $sp-$completions have the following property.

\begin{lemma}
\emph{\textbf{(Uniqueness up to isomorphisms)}}
Let $X$ and $Y$ be integral varieties such that $k(X)=k(Y)$. Then the $sp-$completions $X_{sp}$ and $Y_{sp}$ are essentially equal. In particular,  $X_{sp}$ and $Y_{sp}$ are isomorphic schemes.
\end{lemma}

\begin{proof}
The essential equality is from \emph{Definition 5.7} and \emph{Theorem 5.9}. The isomorphism can be proved in a manner similar to the proof of \emph{Lemma 2.15} in \cite{An4*}.
\end{proof}

\section{Monodromy Action, II}

Let $X$ and $Y$ be two integral varieties. Put $$G_{X}=Aut(X_{sp}/X);$$ $$G_{Y}=Aut(Y_{sp}/Y).$$

Here $(X_{sp},\lambda_{X})$ and $(Y_{sp},\lambda_{Y})$ are $sp-$completions of $X$ and $Y$, respectively.

Now we can give the monodromy actions on $sp-$compltetions.

\begin{lemma}
\emph{\textbf{(Monodromy action)}}
Suppose that there is a group homomorphism $\sigma: G_{X} \to G_{Y}$. Then there is a bijection $$\tau :Hom\left( X,Y\right) \to Hom\left( X_{sp},Y_{sp}\right) , f\mapsto f_{sp}$$ between sets given in  a canonical manner:
\begin{itemize}
\item Let $f\in Hom\left( X,Y\right)$. Then the map $$g\left( x_{0}\right)
\longmapsto \sigma \left( g\right) \left( h\left( x_{0}\right) \right) $$
defines a morphism $$f_{sp}:X_{sp}\rightarrow Y_{sp}$$
for any $x_{0}\in X$ and any $g\in G_{X} .$

\item Let $f_{sp}\in Hom\left( X_{sp},Y_{sp}\right)$. Then the map $$
\lambda_{X}\left( x\right) \longmapsto \lambda_{Y}\left( f_{sp}\left( x\right) \right)
$$ defines a morphism $$f:X\rightarrow Y$$ for any $x\in X_{sp}.$
\end{itemize}
In particular, we have $$f\circ \lambda_{X}=\lambda_{Y}\circ f_{sp}.$$
\end{lemma}

\begin{proof}
It is immediate from \emph{Theorem 5.9} and \emph{Lemma 3.3}.
\end{proof}

\begin{lemma}
\emph{\textbf{($sp-$completion of rational maps, I)}} Let $k(X)=k(Y)$. Then there is a bijection $\tau$ from $Hom(X,Y)$ onto $Hom(X_{sp},Y_{sp})$ given in a canonical manner.

 In particular, $Hom(X,Y)$ must be a non-void set.
\end{lemma}

\begin{proof}
By \emph{Theorem 5.9} we have $$Aut(X_{sp}/X)\cong Gal(k(X)^{al}/k(X)) \cong Aut(Y_{sp}/Y).$$ It follows that there is a group isomorphism $\sigma: G_{X}\cong G_{Y}$. Then it is immediate from \emph{Lemma 6.1}.
\end{proof}

\begin{lemma}
\emph{\textbf{($sp-$completion of rational maps, II)}} Suppose $k(X) \supsetneqq k(Y)$. Then there is a bijection $\tau$ from $Hom(X,Y)$ onto $Hom(X_{sp},Y_{sp})$ given in a canonical manner.

In particular, $Hom(X,Y)$ must be a non-void set and there is a homomorphism $\sigma: G_{X}\to G_{Y}$.
\end{lemma}

\begin{proof}
By \emph{Theorem 5.9} we have $$Aut(X_{sp}/X)\cong Gal(k(X)^{al}/k(X));$$ $$Aut(Y_{sp}/Y)\cong Gal(k(Y)^{al}/k(Y)).$$
As $k(X) \supsetneqq k(Y)$, we have $k(X)^{al} \supseteq k(Y)^{al}$. In particular, $k(X)^{al}$ is a Galois extension over $k(Y)^{al}$. There is a homomorphism from ${Gal(k(X)^{al}/k(X))} $ onto $Gal(k(Y)^{al}/k(Y))$, which  is the composite of the maps $$Gal(k(X)^{al}/k(X)) \to {Gal(k(X)^{al}/k(Y))}$$
and
$$\frac {Gal(k(X)^{al}/k(Y))} {Gal(k(X)^{al}/k(Y)^{al})}\cong Gal(k(Y)^{al}/k(Y)).$$

It follows that there is a homomorphism $\sigma: G_{X}\to G_{Y}$. Then it is immediate from \emph{Lemma 6.1}.
\end{proof}

\begin{remark}
The $sp-$completions of rational maps between integral schemes, as stated above, can be regarded as a generalization of the correspondences between dominant rational maps of algebraic varieties and homomorphisms of algebras in the classical algebraic geometry.
\end{remark}

\section{\emph{qc} Fundamental Groups}

To prove the main theorems of the paper, we also need some results on the \emph{qc} fundamental group of an arithmetic scheme. See \cite{An5} for details.

\subsection{Definition for \emph{qc} fundamental groups}

Let $X$ be an arithmetic variety. Fixed an algebraically closed field $
\Omega $ that contains  $k\left( X\right) $. Here, $\Omega $ is not necessarily algebraic over $k\left(
X\right) .$

Define $X_{qc}\left[ \Omega \right] $ to be the set of arithmetic varieties $Z$ satisfying the two conditions:
\begin{itemize}
\item $Z$ has a reduced
affine covering with values in $\Omega $;

\item There is a
surjective morphism $f:Z\rightarrow X$ of finite type such that $Z$ is
quasi-galois closed over $X.$
\end{itemize}

Naturally there is a partial order $\leq$ in the set $X_{qc}\left[ \Omega \right] $ given in such
a manner:
\begin{itemize}
\item For any $Z_{1},Z_{2}\in X_{qc}\left[ \Omega \right] ,$ we say
\begin{equation*}
Z_{1}\leq Z_{2}
\end{equation*}
if there is a surjective morphism $\varphi :Z_{2}\rightarrow Z_{1}$ of
finite type such that $Z_{2}$ is quasi-galois closed over $Z_{1}.$
\end{itemize}

By \emph{Lemmas 3.6,3.8-10} in \cite{An5}, it is seen that $X_{qc}\left[ \Omega
\right] $ is a directed set and
\begin{equation*}
\{Aut\left( Z/X\right) :Z\in X_{qc}\left[ \Omega \right] \}
\end{equation*}
is an inverse system of groups.

The
inverse limit
\begin{equation*}
\pi _{1}^{qc}\left( X;\Omega \right) \triangleq {\lim_{\longleftarrow}}
_{Z\in X_{qc}\left[ \Omega \right] }{Aut\left( Z/X\right)}
\end{equation*}
of the inverse system $\{Aut\left( Z/X\right) :Z\in X_{qc}\left[ \Omega
\right] \}$ of groups is said to be the \textbf{\emph{qc} fundamental group} of the
scheme $X$ with coefficient in $\Omega .$

\subsection{Main result for \emph{qc} fundamental groups}

There are the following result for \emph{qc} fundamental groups.

\begin{lemma}
\emph{(\cite{An5})}
Let $X$ be an arithmetic variety. Suppose that  $\Omega$ is an algebraically closed field containing  $k\left( X\right) $.
There are the
following statements.
\begin{itemize}
\item There is a group isomorphism
\begin{equation*}
\pi _{1}^{qc}\left( X;\Omega \right) \cong Gal\left( {\Omega }/k\left(
X\right) \right) .
\end{equation*}

\item Take any geometric point $s$ of $X$ over $\Omega $. Then
there is a group isomorphism
\begin{equation*}
\pi _{1}^{et}\left( X;s\right) \cong \pi _{1}^{qc}\left( X;\Omega \right)
_{et}
\end{equation*}
where $\pi _{1}^{qc}\left( X;\Omega \right) _{et}$ is a subgroup of $\pi
_{1}^{qc}\left( X;\Omega \right) $. In particular, $\pi _{1}^{qc}\left( X;\Omega
\right) _{et}$ is a normal subgroup of $\pi _{1}^{qc}\left( X;\Omega \right)
$.
\end{itemize}
\end{lemma}

\section{Proofs of the Main Theorems}

\subsection{Preparatory lemmas}

Let's first prove the below result on the surjection of the sets that are considered.

Let $X$ and $Y$ be arithmetic varieties.

\begin{lemma}
\emph{\textbf{($sp-$completion of rational maps)}} Assume $k(X)\supseteq k(Y)$. Then
$Hom(X,Y)$ must be a non-void set.
\end{lemma}

\begin{proof}
Using the \emph{sp}-completions $(X_{sp},\lambda_{X})$ and $(Y_{sp},\lambda_{Y})$ of $X$ and $Y$, respectively.  Then $k(X_{sp})$ (resp. $k(Y_{sp})$)is the algebraic closure of $k(X)$ (resp. $k(Y)$).

As $k(X)\supseteq k(Y)$, we have $\overline{k(X)}\supseteq \overline{k(Y)}$. Then it is seen that the ring $B$ of an affine open set $V$ in $Y_{sp}$ must be embedded into the ring $A $ of some certain affine open set $U$ in $X_{sp}$; conversely, each $A$ must contain some $B$. It follows that there is a homomorphism $$f_{U}:U=Spec(A)\to V=Spec(B)$$ defined by the inclusion. This gives us a scheme homomorphism $$f_{sp}:X_{sp}\to Y_{sp}.$$

By the projections $\lambda_{X}:X_{sp}\to X$ and $\lambda_{Y}:Y_{sp}\to Y$ we have a unique homomorphism $f:X \to Y$ satisfying the condition $$\lambda_{sp}\circ f_{sp}=f \circ \lambda_{sp}.$$
This completes the proof.
\end{proof}

\begin{lemma}
 Suppose $k\left( Y\right) \subseteq k\left( X\right)$. Then  each element of the set $Hom(\pi_{1}^{et}(X),\pi_{1}^{et}(Y))$ and of the set $Hom(\pi_{1}^{et}(k(X)),\pi_{1}^{et}(k(Y))$ gives an element of the set $Hom (X,Y)$ in a canonical manner, respectively.
\end{lemma}

\begin{proof}
Let $\delta $ be a homomorphism from $\pi_{1}^{et}(X)$ into $ \pi_{1}^{et}(Y)$.

As $k(X) \supsetneqq k(Y)$, by \emph{Lemmas 6.2-3} it is seen that there is a group homomorphism $\sigma$ from $G_{X}=Aut(X_{sp}/X)$ into $G_{Y}=Aut(Y_{sp}/Y)$.

From \emph{Lemma 7.1} it is seen that $k(X)^{un}/k(X)$ and $k(Y)^{un}/k(Y)$ are both Galois extensions.

Then we have
$$\pi_{1}^{et}(X)\cong \frac {Gal(k(X)^{al}/k(X))} {Gal(k(X)^{al}/k(X)^{un})};$$
$$\pi_{1}^{et}(Y)\cong \frac {Gal(k(Y)^{al}/k(Y))} {Gal(k(Y)^{al}/k(Y)^{un})}.$$

It is easily seen that the homomorphisms $\delta$ and $\sigma$ are compatible in a canonical manner.
From \emph{Lemmas 6.2-3} it is seen that for the homomorphism $\delta$ there is a corresponding morphism $f:X \to Y$ which is given in a canonical manner.
\end{proof}

\begin{lemma}
Every morphism $f:X\to Y$ arise from a morphism $f_{qc}:X_{qc}\to Y_{qc}$ of integral schemes given in such a manner:
$$f\circ \phi_{X}=\phi_{Y}\circ f_{qc}.$$
In particular, for the function fields, we have $$k(X) \subseteq k(X_{qc}); k(Y) \subseteq k(Y_{qc}).$$

 Here, $X_{qc}$ is quasi-galois closed over $X$ by a surjective morphism $\phi_{X}$; $Y_{qc}$ is quasi-galois closed over $Y$ by a surjective morphism $\phi_{Y}$ \emph{(see} \S 3\emph{)}.
\end{lemma}

Such an integral scheme $X_{qc}$ is said to be a \textbf{quasi-galois closed cover} of $X$, denoted by $(X_{qc},\phi_{X})$.

\begin{proof}
Just repeat the procedure for a universal cover in \cite{An3} or as in the previous section \S 3.2 of the present paper.
\end{proof}

\begin{lemma}
The \emph{sp}-completions $X_{sp}$ and $Y_{sp}$ are large enough for the morphisms from $X$ into $Y$. That is, there is a surjection from $Hom(X_{sp},Y_{sp})$ onto $Hom(X,Y)$.
\end{lemma}

\begin{proof}
Let $f\in Hom(X,Y)$.
Hypothesize that $f$ does not arise from any morphism $f_{sp}:X_{sp}\to Y_{sp}$ in a canonical manner.

Suppose that $f$ arises from a morphism $$h_{qc}:W_{qc}\to Z_{qc}$$ between quasi-galois closed covers, where $h_{qc}$ is given in a canonical manner by \emph{Lemma 8.3}.

Consider the graphs of schemes (see \emph{\S 5}). It is seen that $\Gamma(X_{sp})$ must be contained in $\Gamma(W_{qc})$.

There are two cases.

\emph{Case (i)}: Assume $\dim X_{sp}=\dim W_{qc}$.

Assume $W_{qc}={(W_{qc})}_{sp}$ without loss of generality.

The function fields $k(X_{sp})$ and $k(W_{qc})$ are two algebraic closure of the field $k(X)$ and hence are isomorphic over $k(X)$.

It is seen that $W_{qc}$ and $ X_{sp}$ are isomorphic schemes over $X$ by the construction for \emph{sp-}completion.
Hence, $f$ arises from $h:X_{sp}\to Y_{sp}$, where there will be in contradiction.

\emph{Case (ii)}: Suppose $\dim X_{sp}<\dim W_{qc}$.

First consider the commutative diagrams which are all given in a canonical manner:

$$h_{qc}\circ \lambda_{W_{qc}}=\lambda_{Z_{qc}}\circ {(h_{qc})}_{sp}: {(W_{qc})}_{sp}\to Z_{qc};$$
$$f\circ \phi_{X}=\phi_{Y}\circ h_{qc}:W_{qc}\to Y.$$

Then consider the commutative diagrams which are all given in a canonical manner:

$$h \circ \lambda_{X_{sp}}=\phi_{Y_{sp}}\circ {(h_{qc})}_{sp}: {(W_{qc})}_{sp}\to Y_{sp};$$
$$f\circ \lambda_{X}=\lambda_{Y}\circ h:X_{sp}\to Y,$$
where $h:X_{sp}\to Y_{sp}$ is uniquely defined in a canonical manner.

It follows that $f$ arises from $h$, where there will be in contradiction.

This completes the proof.
\end{proof}

\subsection{Proofs of the main theorems}

Now we can give the proofs of the  main theorems in the paper.

\begin{proof}
\textbf{(Proof of Theorem 1.1)} By \emph{Lemmas 8.1-4} it is seen that there is a surjection $$t: Hom(\pi_{1}^{et}(k(X)),\pi_{1}^{et}(k(Y))\to Hom(X,Y).$$

Let $\pi$ be the projection from the set $$Hom(\pi_{1}^{et}(k(X)),\pi_{1}^{et}(k(Y))$$ onto the sets $$Hom_{\pi _{1}^{et}(k(X)),\pi _{1}^{et}(k(Y))
}^{out}\left( \pi _{1}^{et}\left( X\right) ,\pi _{1}^{et}\left( Y\right)
\right)$$  given by $f \mapsto [f]$.

From the maps $t$ and $\pi$, we have a map $$\xi: Hom_{\pi _{1}^{et}(k(X)),\pi _{1}^{et}(k(Y))
}^{out}\left( \pi _{1}^{et}\left( X\right) ,\pi _{1}^{et}\left( Y\right)
\right)\to Hom(X,Y)$$ given by $$[f]\mapsto t(f).$$

It is clear that $\xi$ is a surjection of sets.

In the following we prove that $\xi$ is an injection.

In fact, according to the properties of quasi-galois closed schemes (see \emph{Lemmas 3.2-4}),
we have $$Gal(k(X_{sp})/k(X))\cong Aut(X_{sp}/X);$$ $$Gal(k(Y_{sp})/k(Y))\cong Aut(Y_{sp}/Y);$$ $$\pi_{1}^{et}(X)\cong Gal(k(X_{\Omega _{et}})/k(X))\cong Aut(X_{\Omega _{et}}/X);$$  $$\pi_{1}^{et}(Y)\cong Gal(k(Y_{\Omega _{et}})/k(Y))\cong Aut(Y_{\Omega _{et}}/Y).$$

By \emph{Lemma 4.3} and \emph{Lemma 6.1}, it is seen that there are the following monodromy actions:

For the scheme $X$, we have
\begin{itemize}
\item the monodromy action of $Aut(X_{\Omega_{et}}/X)$  on the universal cover $X_{\Omega_{et}}$;

\item the monodromy action of $Aut(X_{sp}/X)$ on the \emph{sp}-completion $X_{sp}$.
\end{itemize}

For the scheme $Y$, we have
\begin{itemize}
\item the monodromy action of $Aut(Y_{\Omega_{et}}/Y)$  on the universal cover $Y_{\Omega_{et}}$;

\item the monodromy action of $Aut(Y_{sp}/Y)$ on the \emph{sp}-completion $Y_{sp}$.
\end{itemize}

It is seen  that there is a
 bijection $$ Hom(\frac{Aut(X_{sp}/X)}{Aut(X_{\Omega _{et}}/X)},\frac{Aut(Y_{sp}/Y)}{Aut(Y_{\Omega _{et}}/Y)})\to Hom(X,Y)$$ between sets.

 In deed, take any $f\in Hom(X,Y)$. There is an $f_{sp}\in Hom(X_{sp},Y_{sp})$ which produces $f$ in a canonical manner. Then $f_{sp}$ produces canonically an $f_{et}\in Hom(X_{\Omega_{et}}, Y_{\Omega_{et}})$. It follows that all elements of $Hom(X,Y)$ arise from the elements of $f_{et}\in Hom(X_{\Omega_{et}}, Y_{\Omega_{et}})$.

 On the other hand, different elements of the set $$Hom(\frac{Aut(X_{sp}/X)}{Aut(X_{\Omega _{et}}/X)},\frac{Aut(Y_{sp}/Y)}{Aut(Y_{\Omega _{et}}/Y)})$$ produce different elements of the set $$Hom(X_{\Omega_{et}}, Y_{\Omega_{et}})$$ and then different elements of $$Hom(X,Y)$$ in a canonical manner, respectively, by the monodromy actions.

Hence, $\xi$ is a bijection.
This completes the proof.
\end{proof}

\begin{proof}
\textbf{(Proof of Theorem 1.2)}
It is immediate from  \emph{Theorem 1.1}.
\end{proof}

\begin{proof}
\textbf{(Proofs of Theorem 1.3-4)}
It is immediate from  the following fact that
$$G(k(X))^{un}\cong \pi^{et}_{1}(X)$$ holds for any arithmetic variety $X$ (see \emph{Lemma 4.1}).
\end{proof}

\begin{remark}
In the above we indeed have proved that there is a bijection $$Hom(X,Y) \cong Hom(\frac{Aut(X_{sp}/X)}{Aut(X_{\Omega _{et}}/X)},\frac{Aut(Y_{sp}/Y)}{Aut(Y_{\Omega _{et}}/Y)})$$ between sets. This is the key point of the section conjecture.
\end{remark}

\end{document}